\documentclass{amsart}
\usepackage{graphicx}

\title[On Lagrangian embeddings of closed non-orientable $3$-manifolds]{On Lagrangian embeddings of closed non-orientable $3$-manifolds}
\author{Toru Yoshiyasu}
\address{Center for Genomic Medicine, Graduate School of Medicine, Kyoto University, 53 Shogoinkawahara-cho, Sakyo-ku, Kyoto-City, Kyoto, 606-8507 Japan}
\email{tyoshiyasu@genome.med.kyoto-u.ac.jp}
\date{\today}

\theoremstyle{plain}
\newtheorem{theorem}{Theorem}[section]
\newtheorem{prop}[theorem]{Proposition}
\newtheorem{lemma}[theorem]{Lemma}

\theoremstyle{definition}
\newtheorem{definition}[theorem]{Definition}

\newtheorem{remark}[theorem]{Remark}

\begin{document}

\begin{abstract}
We prove that for any compact orientable connected $3$-manifold with torus boundary, a concatenation of it and the direct product of the circle and the Klein bottle with an open $2$-disk removed admits a Lagrangian embedding into the standard symplectic $6$-space.
Moreover, minimal Maslov number of the Lagrangian embedding is equal to $1$.
\end{abstract}

\maketitle

\section{Introduction}

In this paper, we study the existence problem of a Lagrangian embedding into the standard symplectic space $\mathbb{R}^{6}_{\mathrm{st}}=(\mathbb{R}^{6},\sum_{j=1}^{3}dx_j\wedge dy_j)$.
Starting with M.~Gromov's discovery of the technique of pseudo-holomorphic curves~\cite{Gr85}, a number of necessary conditions for the existence of a Lagrangian embedding have been proven.
A typical example for the standard symplectic space $\mathbb{R}^{6}_{\mathrm{st}}$ is a partial classification of Lagrangian submanifolds proved by K.~Fukaya: a closed orientable connected prime $3$-manifold $L$ admits a Lagrangian embedding into $\mathbb{R}^{6}_{\mathrm{st}}$ if and only if there exists a non-negative integer $g$ such that $L$ is diffeomorphic to the product $S^{1}\times\Sigma_{g}$, where $\Sigma_{g}$ is the closed orientable connected surface of genus $g$~\cite{Fu06}.
On the other hand, few sufficient conditions for the existence of a Lagrangian embedding were known.
Recently, Y.~Eliashberg and E.~Murphy established the resolving theory of Lagrangian intersections and proved the $h$-principle for Lagrangian embeddings with a concave loose Legendrian boundary~\cite{EM13}.
This $h$-principle has applications to the existence of a Lagrangian embedding of closed manifolds.
For the standard symplectic space $\mathbb{R}^{6}_{\mathrm{st}}$, T.~Ekholm, Y.~Eliashberg, E.~Murphy, and I.~Smith gave the following application: for a closed orientable connected $3$-manifold $L$, there exists a Lagrangian embedding of the connected sum $L\#(S^{1}\times S^{2})$ into $\mathbb{R}^{6}_{\mathrm{st}}$~\cite{EEMS13}.
The theory of loose Legendrian embeddings developed by E.~Murphy~\cite{Mu12} played a central role in the resolving theory and in the application.
Here we give another application of the results of \cite{EM13}.

We introduce some notations and conventions before the statement.
For a non-negative integer $g$, we denote by $N_{2g}$ the closed non-orientable connected surface of Euler characteristic $-2g$.
We fix an embedded closed $2$-disk $D^2_N$ in the Klein bottle $N_0$.
We also fix an identification of the compact surface $N_0\setminus\mathrm{Int}\,D^2_N$ with the compact surface obtained by the orientation-reversing $0$-surgery on the unit closed $2$-disk $D^2$.
This identification induces a diffeomorphism $\partial(S^1\times(N_0\setminus\mathrm{Int}\,D^2_N))\to\partial(S^1\times D^2)$ between $2$-tori.
For a closed orientable connected $3$-manifold $M$ and an embedded $2$-torus $T\subset M$ bounding a solid torus with a parameterization $S^1\times D^2$, we denote by $M_T$ the closed non-orientable connected $3$-manifold
\[
(M\setminus\mathrm{Int}\,(S^1\times D^2))\cup(S^1\times(N_0\setminus\mathrm{Int}\,D^2_N))
\]
concatenated along their boundaries by the above diffeomorphism.
Our main result is the following.
\begin{theorem}\label{Main}
Let $M$ be a closed orientable connected $3$-manifold and $T\subset M$ an embedded $2$-torus bounding a solid torus with a parameterization $S^1\times D^2$.
Then, there exists a Lagrangian embedding $M_T\to\mathbb{R}^{6}_{\mathrm{st}}$ of minimal Maslov number $1$.
In particular, for a closed orientable connected $3$-manifold $L$ and a non-negative integer $g$, there exists a Lagrangian embedding $L\#(S^1\times N_{2g})\to\mathbb{R}^{6}_{\mathrm{st}}$ of minimal Maslov number $1$.
\end{theorem}

Our proof is similar to that of the above application~\cite{EEMS13}, concatenating a Lagrangian filling and a Lagrangian cap.
The existence of a Lagrangian cap is a consequence of the results of \cite{EM13}.
In \cite{EEMS13}, a Lagrangian filling of a loose Legendrian $2$-sphere is constructed.
The new part of this paper is a construction of a Lagrangian filling of a loose Legendrian $2$-torus.

\begin{theorem}\label{torusfilling}
A loose Legendrian $2$-torus of vanishing Maslov class in the standard contact space $\mathbb{R}^{5}_{\mathrm{st}}$ admits a Lagrangian filling $S^1\times(N_0\setminus\mathrm{Int}\,D^2_N)$ of minimal Maslov number $1$.
\end{theorem}

\begin{remark}
By Murphy's $h$-principle for loose Legendrian embeddings~\cite{Mu12}, a loose Legendrian $2$-torus of vanishing Maslov class in the standard contact space $\mathbb{R}^{5}_{\mathrm{st}}$ is unique up to Legendrian isotopy, see \cite[Appendix A]{Mu12}.
\end{remark}

\subsection*{Acknowledgement}

The author is deeply grateful to Yasha Eliashberg for sharing his idea and for helpful discussions.
The construction used in this paper is based on his advice.
The author is also grateful to the referee for valuable comments on improvements of the result and on expositions.
The author is thankful to Emmy Murphy, Morimichi Kawasaki, and Kaoru Ono for helpful conversations.

\section{Proof of Theorems}

\subsection{Construction of a Lagrangian filling}\label{prooffilling}

In this section, we prove Theorem~\ref{torusfilling}.
First, we recall some background on Lagrangian cobordisms.

\begin{definition}
Let $(Y,\alpha)$ be a coorientable contact manifold and $\mathbb{R}\times Y$ its symplectization equipped with the symplectic structure $d(e^{t}\alpha)$, where $t$ is the coordinate of $\mathbb{R}$.
For Legendrian submanifolds $\Lambda_-$ and $\Lambda_+$ of $(Y,\alpha)$, a \textit{Lagrangian cobordism} from $\Lambda_-$ to $\Lambda_+$ is a properly embedded Lagrangian submanifold $L$ in the symplectization $\mathbb{R}\times Y$ such that
\[
L\cap(-\infty,t_-]\times Y=(-\infty,t_-]\times\Lambda_- \text{ and } L\cap[t_+,\infty)\times Y=[t_+,\infty)\times\Lambda_+
\]
for real constants $t_-$ and $t_+$ with $t_-<t_+$.
A Lagrangian cobordism is \textit{exact} if the Lagrangian submanifold $L$ is exact, i.e.~the $1$-form $e^{t}\alpha\mid_L$ is exact.
A Lagrangian cobordism is called a \textit{Lagrangian cap} (resp.~\textit{Lagrangian filling}) if $\Lambda_{+}=\emptyset$ (resp.~$\Lambda_{-}=\emptyset$).
Immersed Lagrangian cobordism, cap, and filling and their exactness are defined in a similar way.
\end{definition}

In the rest of the paper, we identify a Lagrangian cobordism $L$ with its restriction $L\cap[t_-,t_+]\times Y$.
In the case that the constants $t_\pm$ are fixed, we also call $L$ a Lagrangian cobordism from $\{t_-\}\times\Lambda_-$ to $\{t_+\}\times\Lambda_+$.

\begin{definition}
Let $L_0$ be a Lagrangian cobordism from $\{t_0\}\times\Lambda_0$ to $\{t_1\}\times\Lambda_1$ and $L_1$ a Lagrangian cobordism from $\{t_1\}\times\Lambda_1$ to $\{t_2\}\times\Lambda_2$.
A \textit{concatenation} of $L_0$ and $L_1$ along $\{t_1\}\times\Lambda_1$ is a Lagrangian cobordism from $\{t_0\}\times\Lambda_0$ to $\{t_2\}\times\Lambda_2$ defined by the union $L_0\cup L_1$.
A concatenation of immersed Lagrangian cobordisms is defined in a similar way.
\end{definition}

For integers $m=1$ and $2$, we denote by $\mathbb{R}^{2m+1}_{\mathrm{st}}$ the standard contact space $(\mathbb{R}^{2m+1},\alpha_{\mathrm{st}}=dz-\sum_{j=1}^{m}y_{j}dx_j)$ and by $\mathbb{R}\times\mathbb{R}^{2m+1}_{\mathrm{st}}$ its symplectization equipped with the symplectic structure $d(e^t\alpha_{\mathrm{st}})$.
We fix the parameterizations of the standard Legendrian unknot 
\[
K_1\colon S^1\to\mathbb{R}^{3}_{\mathrm{st}}:\theta\mapsto\Bigl(\sin\theta,-\sin2\theta,\frac{2}{3}\cos^3\theta\Bigr)
\]
and its stabilization
\[
K_2\colon S^1\to\mathbb{R}^{3}_{\mathrm{st}}:\theta\mapsto\Bigl(\sin\theta,\sin4\theta,\frac{4}{3}\cos^3\theta-\frac{8}{5}\cos^5\theta\Bigr),
\]
see Figure~\ref{stab}.
\begin{figure}[h]
	\centering
	\includegraphics[width=10cm,clip]{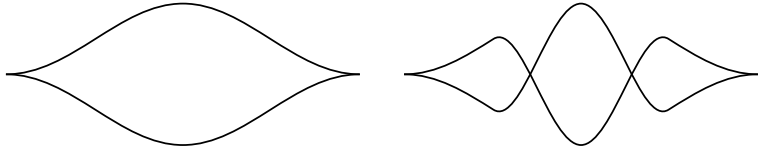}
	\caption{The fronts of the standard Legendrian unknot and its stabilization.}
	\label{stab}
\end{figure}

The following construction of a Lagrangian filling of $K_2$ is the main part of the proof of Theorem~\ref{torusfilling}.

\begin{prop}\label{knotfilling}
There exists a Lagrangian filling $N_0\setminus \mathrm{Int}\,D^2_N\to\mathbb{R}\times\mathbb{R}^{3}_{\mathrm{st}}$ of $K_2$ of minimal Maslov number $1$.
\end{prop}

\begin{proof}
First, we construct an immersed Lagrangian cobordism from $K_1$ to $K_2$ with exactly one double point in $\mathbb{R}\times\mathbb{R}^{3}_{\mathrm{st}}$ as follows.
We pick a smooth cutoff function $\rho_1\colon[0,1]\to[0,1]$ such that
\begin{enumerate}
	\item[$\bullet$] $\rho_1(s)=0$ and $\rho_1(1-s)=1$ if $0\leq s\leq\frac{1}{3}$, and
	\item[$\bullet$] $\rho_1^{\prime}(s)>0$ if $\frac{1}{3}<s<\frac{2}{3}$.
\end{enumerate}
Then, for a positive integer $n$, we define another cutoff function $\rho_2\colon[0,n]\to[0,1]$ by $\rho_2(t)=\rho_1\bigl(\frac{t}{n}\bigr)$.
Using the cutoff function $\rho_2$, we consider the homotopy
\[
k_{\mathrm{fr}}\colon[0,n]\times S^1\to\mathbb{R}^2:(t,\theta)\mapsto\Bigl(\sin\theta,\frac{2}{3}\cos^3\theta+\rho_2(t)\Bigl(\frac{2}{3}\cos^3\theta-\frac{8}{5}\cos^5\theta\Bigr)\Bigr)
\]
from the front of $K_1$ to that of $K_2$.
The monotonicity of $\rho_1$ implies that there is the unique tangent point $(0,0)=k_{\mathrm{fr}}(T,0)=k_{\mathrm{fr}}(T,\pi)$, where $T$ is defined by the equation $\rho_2(T)=\frac{5}{7}$.
Solving the differential equation $y=\frac{dz}{dx}$ on $\mathbb{R}^{3}_{\mathrm{st}}$, the front homotopy $k_{\mathrm{fr}}$ lifts to the Legendrian regular homotopy $f\colon[0,n]\times S^1\to\mathbb{R}^{3}_{\mathrm{st}}$,
\[
f(t,\theta)=\Bigl(\sin\theta,-\sin2\theta+\rho_2(t)(\sin4\theta+\sin2\theta),\frac{2}{3}\cos^3\theta+\rho_2(t)\Bigl(\frac{2}{3}\cos^3\theta-\frac{8}{5}\cos^5\theta\Bigr)\Bigr).
\]
Its trace $\mathrm{Tr}(f)\colon[0,n]\times S^1\to\mathbb{R}\times\mathbb{R}^{3}_{\mathrm{st}}:(t,\theta)\mapsto(t,f(t,\theta))$ has exactly one double point
\[
(T,0,0,0)=\mathrm{Tr}(f)(T,0)=\mathrm{Tr}(f)(T,\pi)
\]
corresponding to the tangent point $(0,0)$.
We can check that its self-intersection number is equal to $-1$.
Perturbing the trace $\mathrm{Tr}(f)$, we construct the Lagrangian immersion $\tilde{f}_{-1}\colon[0,n]\times S^1\to\mathbb{R}\times\mathbb{R}^{3}_{\mathrm{st}}$,
\begin{align*}
	\tilde{f}_{-1}(t,\theta)=\Bigl( & t,\sin\theta,-\sin2\theta+\rho_2(t)(\sin4\theta+\sin2\theta),\\
	& \frac{2}{3}\cos^3\theta+(\rho_2(t)+\rho_{2}^{\prime}(t))\Bigl(\frac{2}{3}\cos^3\theta-\frac{8}{5}\cos^5\theta\Bigr)\Bigr).
\end{align*}
Choosing the integer $n$ sufficiently large, the derivative $\rho_{2}^{\prime}$ can be arbitrarily small, and hence there is one-to-one correspondence between the double point of $\tilde{f}_{-1}$ and that of $\mathrm{Tr}(f)$ preserving self-intersection number.
We denote by $q$ the double point of self-intersection number $-1$ of the Lagrangian immersion $\tilde{f}_{-1}$.
The image $\tilde{f}_{-1}([0,n]\times S^1)$ is an immersed Lagrangian cobordism from $K_1$ to $K_2$, since $\tilde{f}_{-1}([0,\varepsilon)\times S^1)=[0,\varepsilon)\times K_1$ and $\tilde{f}_{-1}((n-\varepsilon,n]\times S^1)=(n-\varepsilon,n]\times K_2$ for a sufficiently small positive constant $\varepsilon$.

We recall that the standard Legendrian unknot $\{0\}\times K_1$ admits a Lagrangian filling by $2$-disk in $(-\infty,0]\times\mathbb{R}^{3}_{\mathrm{st}}$.
Concatenating this Lagrangian filling and the immersed Lagrangian cobordism $\tilde{f}_{-1}$ along $\{0\}\times K_1$, we construct an immersed Lagrangian filling $\tilde{h}_{-1}\colon D^2\to\mathbb{R}\times\mathbb{R}^{3}_{\mathrm{st}}$ of $K_2$ with exactly one double point $q$.
Resolving the double point $q$ by Polterovich's Lagrangian surgery~\cite{Po91}, we obtain a Lagrangian filling $\tilde{h}_{0}\colon N_{0}\setminus\mathrm{Int}\,D^2_N\to\mathbb{R}\times\mathbb{R}^{3}_{\mathrm{st}}$ of $K_2$.
Although there are two choices of the surgery depending on an order of the two sheets at the double point $q$, each choice yields the same result in this dimension, see \cite{Po91}.

Next, we compute the difference of Maslov potentials on the two sheets at the double point $q$ of the Lagrangian cobordism $\tilde{f}_{-1}$ and then show minimal Maslov number of the Lagrangian filling $\tilde{h}_0$ is equal to $1$.
A similar computation was made in \cite[Section~2.2]{ES16}.
It suffices to consider the case $\rho_1\colon[0,1]\to[0,1]$ is the identity and $n=7$.
In fact, although the identity is not a cutoff function, the linear homotopy from the cutoff function $\rho_1$ to the identity can be realized by a Lagrangian regular homotopy of the Lagrangian immersion $\tilde{f}_{-1}$.
Moreover, a change of the integer $n$ can also be realized by a Lagrangian regular homotopy of $\tilde{f}_{-1}$.
A straightforward computation shows that the Lagrangian immersion $\tilde{f}_{-1}$ has exactly one double point $q=\tilde{f}_{-1}(4,0)=\tilde{f}_{-1}(4,\pi)$ if $\rho_1$ is the identity and $n=7$.
In this case, the Lagrangian immersion $\tilde{f}_{-1}\colon[0,7]\times S^1\to\mathbb{R}\times\mathbb{R}^{3}_{\mathrm{st}}$ is of the form
\begin{align*}
	\tilde{f}_{-1}(t,\theta)=\Bigl( & t,\sin\theta,-\sin2\theta+\frac{t}{7}(\sin4\theta+\sin2\theta),\\
	& \frac{2}{3}\cos^3\theta+\frac{t+1}{7}\Bigl(\frac{2}{3}\cos^3\theta-\frac{8}{5}\cos^5\theta\Bigr)\Bigr).
\end{align*}

For the computation, we choose the path $l\colon[0,1]\to[0,7]\times S^1:s\mapsto(4,\pi s)$, the reference Lagrangian subspace $P_0=(\tilde{f}_{-1})_{\ast}T_{(4,0)}([0,7]\times S^1)$, the symplectic structure $d(e^{t}\alpha_{\mathrm{st}})=e^{t}(dt\wedge(dz-ydx)+dx\wedge dy)$ on the symplectization $\mathbb{R}\times\mathbb{R}^{3}_{\mathrm{st}}$, and its symplectic $4$-frame $\{\frac{\partial}{\partial t}, \frac{\partial}{\partial z}, \frac{\partial}{\partial x}+y\frac{\partial}{\partial z},\frac{\partial}{\partial y}\}$.
Then the Lagrangian $2$-frame $\{(\tilde{f}_{-1})_{\ast}\frac{\partial}{\partial t},(\tilde{f}_{-1})_{\ast}\frac{\partial}{\partial\theta}\}$ along the path $(\tilde{f}_{-1}\circ l)([0,1])$ is of the form
\begin{align*}
	(\tilde{f}_{-1})_{\ast}\frac{\partial}{\partial t}= & \frac{\partial}{\partial t}+\frac{1}{7}\Bigl(\frac{2}{3}\cos^3(\pi s)-\frac{8}{5}\cos^5(\pi s)\Bigr)\frac{\partial}{\partial z}+\frac{1}{7}(\sin(4\pi s)+\sin(2\pi s))\frac{\partial}{\partial y},\\
	(\tilde{f}_{-1})_{\ast}\frac{\partial}{\partial\theta}= & \frac{1}{7}\cos(\pi s)(\sin(4\pi s)+\sin(2\pi s))\frac{\partial}{\partial z}+\cos(\pi s)\Bigl(\frac{\partial}{\partial x}+y\frac{\partial}{\partial z}\Bigr)\\
	& +\frac{2}{7}(8\cos(4\pi s)-3\cos(2\pi s))\frac{\partial}{\partial y}.
\end{align*}
Taking their components, we define the paths of matrices
\begin{align*}
	X(s) & =
	\begin{pmatrix}
	1 & 0\\
	0 & \cos(\pi s)
	\end{pmatrix}
	\text{ and}\\
	Y(s) & =
	\begin{pmatrix}
	\frac{1}{7}(\frac{2}{3}\cos^3(\pi s)-\frac{8}{5}\cos^5(\pi s)) & \frac{1}{7}(\sin(4\pi s)+\sin(2\pi s))\\
	\frac{1}{7}\cos(\pi s)(\sin(4\pi s)+\sin(2\pi s)) & \frac{2}{7}(8\cos(4\pi s)-3\cos(2\pi s))
	\end{pmatrix}.
\end{align*}
The difference of Maslov potentials on the two sheets at the double point $q$ for the reference Lagrangian subspace $P_0$ is computed by counting points through the angle
\[
\arg(\det(X(0)+iY(0))^2)=2\arg(\det(X(0)+iY(0)))
\]
on the path $\det(X(s)+iY(s))^2\colon[0,1]\to\mathbb{C}\setminus\{0\}$.
This counting is invariant under a homotopy of paths $[0,1]\to\mathbb{C}\setminus\{0\}$ relative to the boundary.
In order to count, we look at the path $\det(X(s)+iY(s))\colon[0,1]\to\mathbb{C}\setminus\{0\}$ being of the form
\begin{align*}
	\frac{1}{49}\Bigl(\frac{704}{5}\cos^9(\pi s)-\frac{640}{3}\cos^7(\pi s)+\frac{1388}{15}\cos^5(\pi s)-\frac{32}{3}\cos^3(\pi s)+49\cos(\pi s)\Bigr)\\
	+\frac{i}{7}\Bigl(-\frac{8}{5}\cos^6(\pi s)+\frac{386}{3}\cos^4(\pi s)-140\cos^2(\pi s)+22\Bigr).
\end{align*}
Using the expression, we can show that
\begin{enumerate}
	\item $\det(X(0)+iY(0))=\frac{1}{105}(115+136i)$;\label{1st}
	\item $\det(X(\frac{1}{2})+iY(\frac{1}{2}))=\frac{22}{7}i$;
	\item $\det(X(1)+iY(1))=\frac{1}{105}(-115+136i)$;
	\item $\mathrm{Re}(\det(X(s)+iY(s)))=-\mathrm{Re}(\det(X(1-s)+iY(1-s)))$ if $0\leq s\leq\frac{1}{2}$;
	\item $\mathrm{Im}(\det(X(s)+iY(s)))=\mathrm{Im}(\det(X(1-s)+iY(1-s)))$ if $0\leq s\leq\frac{1}{2}$; and\label{last}
	\item $\mathrm{Re}(\det(X(s)+iY(s)))>0$ if $0\leq s<\frac{1}{2}$.\label{posi}
\end{enumerate}
Actually, the equalities (\ref{1st})--(\ref{last}) are straightforward.
The inequality (\ref{posi}) can be shown by using the inequality of arithmetic and geometric means to estimate the first three terms.
These properties (\ref{1st})--(\ref{posi}) imply that there exists a homotopy of paths $[0,1]\to\mathbb{C}\setminus\{0\}$ from $\det(X(s)+iY(s))$ to the counterclockwise circular arc, from $\det(X(0)+iY(0))$ to $\det(X(1)+iY(1))$, relative to the boundary.
Moreover, its rotation angle $\varphi_1$ satisfies $\frac{\pi}{4}<\varphi_1<\frac{\pi}{2}$, see Figure~\ref{determinant}.
\begin{figure}[h]
	\centering
	\includegraphics[width=5cm,clip]{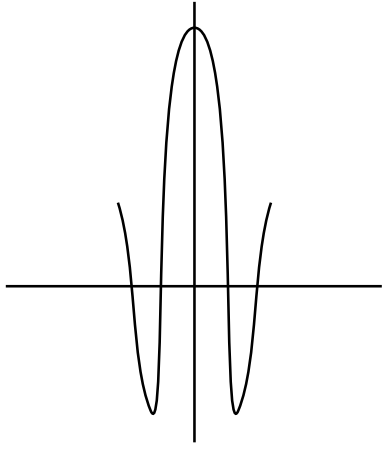}
	\caption{The path $\det(X(s)+iY(s))$ in $\mathbb{C}\setminus\{0\}$.}
	\label{determinant}
\end{figure}
This homotopy induces a homotopy of paths $[0,1]\to\mathbb{C}\setminus\{0\}$ from $\det(X(s)+iY(s))^2$ to the counterclockwise circular arc, from $\det(X(0)+iY(0))^2$ to $\det(X(1)+iY(1))^2$, relative to the boundary.
Furthermore, its rotation angle $2\varphi_1$ satisfies $\frac{\pi}{2}<2\varphi_1<\pi$.
Therefore, the difference of Maslov potentials on the two sheets at the double point $q$ for the reference Lagrangian subspace $P_0$ is equal to $\frac{1}{2}$.

Combining the above discussion with the construction of Polterovich's Lagrangian $1$-handle $[0,1]\times S^1$~\cite{Po91}, we can see that minimal Maslov number of the Lagrangian filling $\tilde{h}_0$ is equal to $1$.
In fact, Polterovich's Lagrangian surgery~\cite{Po91} resolves the double point $q$ of $\tilde{f}_{-1}$ and creates two loops.
The one is the meridian loop $\{0\}\times S^1$ of the Lagrangian $1$-handle whose Maslov index is equal to zero.
The other one is the orientation-reversing loop obtained by smoothing the path $\tilde{f}_{-1}\circ l$ along the path $[0,1]\times\{\text{pt}\}$ of the Lagrangian $1$-handle.
By this smoothing, the path $\det(X(s)+iY(s))^2$ extends to a loop $S^1\to\mathbb{C}\setminus\{0\}$.
The rotation angle $\varphi_2$ of the extended part of this loop $S^1\to\mathbb{C}\setminus\{0\}$ is coming from the path $[0,1]\times\{\text{pt}\}$ of the Lagrangian $1$-handle, and hence satisfies $-2\pi<\varphi_2<2\pi$, see \cite[Section 2]{Po91}.
Therefore, we obtain the estimate
\[
-\frac{3}{2}\pi<2\varphi_1+\varphi_2<3\pi.
\]
We recall that the Maslov index of a loop is odd if and only if the loop is orientation-reversing, so we also have
\[
2\varphi_1+\varphi_2\in\{2\pi(2k+1)\mid k\in\mathbb{Z}\}\text{.}
\]
We conclude that the rotation angle $2\varphi_1+\varphi_2$ of the orientation-reversing loop is equal to $2\pi$, and thus its Maslov index is equal to $1$.
\end{proof}

\begin{proof}[Proof of Theorem \ref{torusfilling}]\label{proofcap}
We claim that the front $S^1$-spinning~\cite{Go14} of the Lagrangian filling $\tilde{h}_0$ constructed in Proposition~\ref{knotfilling} is the desired one.
First, composing a parallel transformation on the $x$-coordinate direction by a sufficiently large positive constant, we modify the Lagrangian filling $\tilde{h}_{0}\colon N_0\setminus\mathrm{Int}\,D^2_N\to\mathbb{R}\times\mathbb{R}^{3}_{\mathrm{st}}:p\mapsto(\tilde{t}(p),\tilde{x}(p),\tilde{y}(p),\tilde{z}(p))$ to satisfy $\tilde{x}>0$.
Applying the front $S^1$-spinning construction~\cite{Go14} to $\tilde{h}_0$, we obtain the Lagrangian filling $F\colon S^{1}\times(N_0\setminus\mathrm{Int}\,D^2_N)\to\mathbb{R}\times\mathbb{R}^{5}_{\mathrm{st}}$,
\[
F(\theta,p)=(\tilde{t}(p),\tilde{x}(p)\cos\theta,\tilde{y}(p)\cos\theta,\tilde{x}(p)\sin\theta,\tilde{y}(p)\sin\theta,\tilde{z}(p)).
\]
Then the Legendrian torus boundary $\Sigma_{S^1}K_2=F(S^{1}\times\partial D^2)\subset\{n\}\times\mathbb{R}^{5}_{\mathrm{st}}$ is the front $S^{1}$-spinning~\cite{EES05} of $K_2$, and hence is loose \cite{DG14} in the sense of \cite{Mu12}.
Moreover, its Maslov class vanishes.
In fact, the Maslov index of the loop $F(S^1\times\{\mathrm{pt}\})$ is equal to zero~\cite{EES05}.
For the loop $F(\{\mathrm{pt}\}\times\partial D^2)=K_2$, the vanishing of the Maslov index is straightforward.
\end{proof}

\begin{remark}
The Lagrangian fillings $\tilde{h}_0$ and $F$ are non-exact since the Legendrian boundary of  $\tilde{h}_0$ is stabilized \cite{Ch02, Ek08} and that of $F$ is loose \cite{Mu12}.
\end{remark}

\subsection{Construction of a Lagrangian embedding}

In this section, we prove Theorem~\ref{Main}.
Concatenating the Lagrangian filling in Theorem~\ref{torusfilling} and a Lagrangian cap, we construct the desired Lagrangian embedding.
We start with a construction of an immersed Lagrangian cap for the particular case.
For a non-negative integer $g$, we fix an embedded closed $2$-disk $D^2_g$ in the closed surface $\Sigma_g$.

\begin{lemma}\label{immcap}
Let $L$ be a closed orientable connected $3$-manifold, $g$ a non-negative integer, and $M^{\prime}$ the connected sum
\[
L\#(S^1\times(\Sigma_g\setminus\mathrm{Int}\,D^2_g)).
\]
Then, the $3$-manifold $M^{\prime}$ can be realized as an immersed Lagrangian cap of the loose Legendrian torus $\Sigma_{S^1}K_2$ and of the self-intersection number zero modulo two.
\end{lemma}

\begin{proof}
We note that the existence of a Lagrangian immersion $M^{\prime}\to\mathbb{R}\times\mathbb{R}^{5}_{\mathrm{st}}$ is equivalent to the triviality of the complexified tangent bundle $TM^{\prime}\otimes\mathbb{C}$ by Gromov--Lees $h$-principle for Lagrangian immersions~\cite{Gr86, Le76}.
In this case, the parallelizability of $M^{\prime}$ implies the latter condition.
Moreover, we can choose a Lagrangian immersion $M^{\prime}\to\mathbb{R}\times\mathbb{R}^{5}_{\mathrm{st}}$ to be an immersed Lagrangian cap of $\Sigma_{S^1}K_2$ as follows.

The parallelizability of $M^{\prime}$ allows us to take a Lagrangian homomorphism $TM^{\prime}\to T(\mathbb{R}\times\mathbb{R}^{5}_{\mathrm{st}})$ such that its Gauss map $M^{\prime}\to\mathrm{U}(3)\slash\mathrm{O}(3)$ is constant, where $\mathrm{U}(3)\slash\mathrm{O}(3)$ is the Lagrangian Grassmannian.
Its restriction on the boundary $\partial M^{\prime}$ is homotopic to the Lagrangian homomorphism $dF\mid_{S^1\times\partial D^2}$ defined by the Lagrangian filling $F$ constructed in Theorem~\ref{torusfilling} as Lagrangian homomorphisms, since the Maslov class of $\Sigma_{S^1}K_2$ vanishes.
Therefore, there exists a Lagrangian homomorphism $\Phi\colon TM^{\prime}\to T([n,\infty)\times\mathbb{R}^{5}_{\mathrm{st}})$ that is an extension of $dF\mid_{S^1\times\partial D^2}$.
We denote by $\phi\colon M^{\prime}\to[n,\infty)\times\mathbb{R}^{5}_{\mathrm{st}}$ the underlying map of $\Phi$.
Using the contractibility of $\mathbb{R}\times\mathbb{R}^{5}_{\mathrm{st}}$, we may assume that the map $\phi$ is a smooth extension of $F\mid_{S^1\times\partial D^2}$.
Then the relative cohomology class $[\phi^{\ast}d(e^t\alpha_{\mathrm{st}})]\in H^2(M^{\prime},\partial M^{\prime};\mathbb{R})$ vanishes by the property $H^2(\mathbb{R}\times\mathbb{R}^{5}_{\mathrm{st}};\mathbb{R})=0$.
By the construction, we can choose the formal Lagrangian immersion $(\phi,\Phi)$ so that
\begin{enumerate}
	\item[$\bullet$] $d\phi=\Phi$ on a small neighborhood $U$ of the boundary $\partial M^{\prime}=S^1\times\partial D^2_g$;
	\item[$\bullet$] $\phi(U)\cap[n,n+\varepsilon)\times\mathbb{R}^{5}_{\mathrm{st}}=[n,n+\varepsilon)\times\Sigma_{S^1}K_2$ for a small positive constant $\varepsilon$.
\end{enumerate}
Applying the relative version of Gromov--Lees $h$-principle for Lagrangian immersions~\cite{Gr86, Le76}, see also \cite[Theorem 16.3.2]{EM02}, to the formal Lagrangian immersion $(\phi,\Phi)$, we obtain an immersed Lagrangian cap $\tilde{\phi}\colon M^{\prime}\to[n,\infty)\times\mathbb{R}^{5}_{\mathrm{st}}$ of $\Sigma_{S^1}K_2$.

We show that the Lagrangian immersion $\tilde{\phi}\colon M^{\prime}\to[n,\infty)\times\mathbb{R}^{5}_{\mathrm{st}}$ can be chosen to have the self-intersection number zero modulo two.
If the self-intersection number of $\tilde{\phi}$ is equal to one modulo two, we modify it as follows.
There exists a Lagrangian immersion $G$ of the $3$-sphere to a Darboux ball in $\mathbb{R}\times\mathbb{R}^{5}_{\mathrm{st}}$ of the self-intersection number zero modulo two~\cite{Au88}.
We may assume that
\begin{enumerate}
	\item[$\bullet$] the image $G(S^3)$ is contained in $(n,\infty)\times\mathbb{R}^{5}_{\mathrm{st}}$;
	\item[$\bullet$] $\tilde{\phi}$ and $G$ intersect transversely at exactly two points.
\end{enumerate}
In fact, we can deform the Lagrangian immersion $G$ to satisfy these conditions by a parallel transformation and a small perturbation as a Lagrangian immersion.
Applying Polterovich's Lagrangian surgery~\cite{Po91} to one intersection, we construct the connected sum $\tilde{\phi}\# G\colon M^{\prime}\to[n,\infty)\times\mathbb{R}^{5}_{\mathrm{st}}$ of the Lagrangian immersions such that
\[
(\tilde{\phi}\# G)(M^{\prime})\cap[n,n+\varepsilon^{\prime}]\times\mathbb{R}^{5}_{\mathrm{st}}=\tilde{\phi}(M^{\prime})\cap[n,n+\varepsilon^{\prime}]\times\mathbb{R}^{5}_{\mathrm{st}}
\]
for a small positive constant $\varepsilon^{\prime}$.
Thus the image $(\tilde{\phi}\# G)(M^{\prime})$ is an immersed Lagrangian cap of $\Sigma_{S^1}K_2$ and its self-intersection number is equal to zero modulo two.
\end{proof}

\begin{proof}[Proof of Theorem~\ref{Main}]
We first prove the particular case.
We denote by $\tilde{\phi}$ the Lagrangian immersion constructed in the proof of Lemma~\ref{immcap}.
We deform the Lagrangian immersion $\tilde{\phi}$ to a formal Lagrangian embedding of $M^{\prime}$ into $[n,\infty)\times\mathbb{R}^{5}_{\mathrm{st}}$ relative to a small neighborhood of the loose Legendrian boundary $\Sigma_{S^1}K_2$.
Applying \cite[Theorem 2.2]{EM13} to this formal Lagrangian embedding, we get a Lagrangian cap $\tilde{\phi}_0\colon M^{\prime}\to[n,\infty)\times\mathbb{R}^{5}_{\mathrm{st}}$ of $\Sigma_{S^1}K_2$.
Concatenating the Lagrangian filling $F$ and the Lagrangian cap $\tilde{\phi}_0$ along $\Sigma_{S^1}K_2$, we obtain a Lagrangian embedding $L\#(S^1\times N_{2g})\to\mathbb{R}\times\mathbb{R}^{5}_{\mathrm{st}}$ of minimal Maslov number $1$.
We recall that the symplectization $\mathbb{R}\times S^{5}_{\mathrm{st}}$ of the standard contact sphere $S^{5}_{\mathrm{st}}$ is symplectomorphic to the symplectic manifold $\mathbb{R}^{6}_{\mathrm{st}}\setminus\{\mathbf{0}\}$.
The construction is done by composing a symplectic embedding $\mathbb{R}\times\mathbb{R}^{5}_{\mathrm{st}}\to\mathbb{R}\times S^{5}_{\mathrm{st}}\subset\mathbb{R}^{6}_{\mathrm{st}}$ induced by a contact embedding $\mathbb{R}^{5}_{\mathrm{st}}\to S^{5}_{\mathrm{st}}$.

We can similarly prove the general case.
In fact, for a closed orientable connected $3$-manifold $M$ and an embedded $2$-torus $T\subset M$ bounding a solid torus with a parameterization $S^1\times D^2$, the $3$-manifold $M\setminus\mathrm{Int}\,(S^1\times D^2)$ is parallelizable and its boundary is diffeomorphic to a $2$-torus.
The proof of Lemma~\ref{immcap} depends only on these properties, so the $3$-manifold $M\setminus\mathrm{Int}\,(S^1\times D^2)$ can also be realized as an immersed Lagrangian cap of $\Sigma_{S^1}K_2$ and of the self-intersection number zero modulo two.
The existence of such an immersed Lagrangian cap allows us to apply the same construction.
\end{proof}

\begin{remark}
Let $N$ be a closed non-orientable connected $3$-manifold with trivial complexified tangent bundle $TN\otimes\mathbb{C}\to N$ and $M$ a $3$-manifold as in Theorem~\ref{Main}.
Then, the connected sum $M_T\#N$ also admits a Lagrangian embedding into $\mathbb{R}^{6}_{\mathrm{st}}$.
Actually, we can construct an immersed Lagrangian cap $(M\setminus\mathrm{Int}\,(S^1\times D^2))\#N$ of $\Sigma_{S^1}K_2$ by taking the connected sum of the Lagrangian cap $M\setminus\mathrm{Int}\,(S^1\times D^2)$ and a Lagrangian immersion $N\to\mathbb{R}\times\mathbb{R}^{5}_{\mathrm{st}}$ in a way similar to the construction of $\tilde{\phi}\#G$ in the proof of Lemma~\ref{immcap}.
The rest of the construction is the same to the proof of Theorem~\ref{Main}.
On the other hand, there exists a compact $3$-manifold with trivial complexified tangent bundle and torus boundary that can not be realized as an immersed Lagrangian cap of $\Sigma_{S^1}K_2$.
The direct product of the circle and the M{\"o}bius band is an example of such a $3$-manifold.
In fact, the boundary of the M{\"o}bius band is homotopic to twice the centered orientation-reversing loop, so its Maslov index can not be zero for any Lagrangian immersion.
\end{remark}

\begin{remark}
By \cite[Theorem 2.2]{EM13}, we can choose the Lagrangian cap $M\setminus\mathrm{Int}\,(S^1\times D^2)$ to be exact.
In particular, Theorem~\ref{Main} for the case $L=S^3$ gives a Lagrangian embedding $S^1\times N_{2g}\to\mathbb{R}\times\mathbb{R}^{5}_{\mathrm{st}}$ being a concatenation of the Lagrangian filling $F$ and an exact Lagrangian cap along the loose Legendrian torus $\Sigma_{S^1}K_2$.
The particular case of Theorem~\ref{Main} is a consequence of the existence of this Lagrangian embedding.
On the other hand, if $g\geq 1$ then no Lagrangian embedding $S^{1}\times\Sigma_{g}\to\mathbb{R}\times\mathbb{R}^{5}_{\mathrm{st}}$ can be a concatenation of a Lagrangian filling and an exact Lagrangian cap by \cite[Proposition 1.4]{Di15} and \cite[Theorem 1.1]{BM14}.
\end{remark}

\end{document}